\documentclass[11pt]{amsart}

\usepackage{amsmath, amssymb, enumerate, amsthm, setspace, tikz-cd, hyperref}
\usepackage[shortcuts]{extdash}
\usepackage[all]{xy}
\theoremstyle{plain}
\newtheorem{theorem}{Theorem}[section]
\newtheorem{lemma}[theorem]{Lemma}
\newtheorem{proposition}[theorem]{Proposition}
\newtheorem{corollary}[theorem]{Corollary}

\theoremstyle{definition}
\newtheorem{definition}[theorem]{Definition}

\newtheorem*{acknowledgements}{Acknowledgements}

\DeclareMathOperator{\tr}{tr}

\DeclareMathOperator{\Ext}{Ext}
\DeclareMathOperator{\Hom}{Hom}

\begin{document}

\title{A new proof of the Tikuisis-White-Winter Theorem}
\author{Christopher Schafhauser}
\address{Department of Pure Mathematics, University of Waterloo, 200 University Avenue West, Waterloo, ON, Canada, N2L 3G1}
\email{cschafhauser@uwaterloo.ca}
\date{\today}
\begin{abstract}
  A trace on a C$^*$-algebra is amenable (resp. quasidiagonal) if it admits a net of completely positive, contractive maps into matrix algebras which approximately preserve the trace and are approximately multiplicative in the 2-norm (resp. operator norm).  Using that the double commutant of a nuclear C$^*$-algebra is hyperfinite, it is easy to see that traces on nuclear C$^*$-algebras are amenable.  A recent result of Tikuisis, White, and Winter shows that faithful traces on separable, nuclear C$^*$-algebras in the UCT class are quasidiagonal.  We give a new proof of this result using the extension theory of C$^*$-algebras and, in particular, using a version of the Weyl-von Neumann Theorem due to Elliott and Kucerovsky.
\end{abstract}
\maketitle

\section{Introduction}

A trace on a unital C$^*$-algebra $A$ is a linear functional $\tau$ such that $\tau(1) = \|\tau\| = 1$ and $\tau(ab) = \tau(ba)$ for all $a, b \in A$.  The study of traces on C$^*$-algebras and von Neumann algebras has a long history.  In the case $A$ is commutative, the Riesz Representation Theorem yields a canonical bijection between traces on $A$ and regular, Borel, probability measures on the Gelfand spectrum of $A$.  In the non-commutative setting, studying traces gives, in some sense, connections between the topological and measurable sides of operator algebras.  In particular, every trace on a C$^*$-algebra induces a finite von Neumann algebra through the GNS construction, and this leads to rich connections between C$^*$-algebras and von Neumann algebras.

In Connes's seminal work on injective factors in \cite{Connes:InjectiveFactors}, a more refined study of traces surfaced.  Connes proved that a $\mathrm{II}_1$-factor with separable predual is injective if, and only if, the trace admits certain finite-dimensional approximations -- in modern language, the trace is \emph{amenable}.  Amenable traces were studied further by Kirchberg in \cite{Kirchberg:AmenableTraces}, and later, several other approximation properties of a similar nature were introduced by N. Brown in \cite{BrownQDTraces}.  We now formally define these approximation conditions.

\begin{definition}\label{defn:ApproxOfTraces}
Let $\tr_k$ denote the unique trace on $\mathbb{M}_k$, the algebra of $k \times k$ matrices over $\mathbb{C}$, and, for all $a \in \mathbb{M}_k$, define $\|a\|_2 = \tr_k(a^*a)^{1/2}$.  Suppose $\tau$ is a trace on a separable C$^*$-algebra $A$.
\begin{enumerate}
  \item We call $\tau$ \emph{amenable} if for each integer $n \geq 1$, there is an integer $k(n) \geq 1$ and a completely positive, contractive map $\varphi_n : A \rightarrow \mathbb{M}_{k(n)}$ such that
      \[\lim_{n \rightarrow \infty} \| \varphi_n(ab) - \varphi_n(a) \varphi_n(b) \|_2 = 0, \quad \text{and} \quad \lim_{n \rightarrow \infty} \tr_{k(n)}(\varphi_n(a)) = \tau(a) \]
      for all $a, b \in A$.
  \item We call $\tau$ \emph{quasidiagonal} if for each integer $n \geq 1$, there is an integer $k(n) \geq 1$ and a completely positive, contractive map $\varphi_n : A \rightarrow \mathbb{M}_{k(n)}$ such that
      \[\lim_{n \rightarrow \infty} \| \varphi_n(ab) - \varphi_n(a) \varphi_n(b) \| = 0, \quad \text{and} \quad \lim_{n \rightarrow \infty} \tr_{k(n)}(\varphi_n(a)) = \tau(a) \]
      for all $a, b \in A$.
\end{enumerate}
\end{definition}

Note that for each integer $k \geq 1$ and each $x \in \mathbb{M}_k$, we have
\[ \|x\|_2 \leq \|x\| \leq \sqrt{k} \|x\|_2. \]
The first inequality implies quasidiagonal traces are amenable.  A question of N.\ Brown asks whether every amenable trace is quasidiagonal (see \cite{BrownQDTraces}).  Substantial progress on this problem has recently been made by Tikuisis, White, and Winter in \cite{TikuisisWhiteWinter}.  We refer the reader to \cite{BrownQDTraces} and \cite{TikuisisWhiteWinter} for a further discussion on the classes of amenable and quasidiagonal traces.  In this paper, we provide a self-contained proof of the the main result of \cite{TikuisisWhiteWinter}.  In fact, we prove the following more general result given by Gabe in \cite{Gabe:TWW}.

\begin{theorem}[see \cite{Gabe:TWW} and \cite{TikuisisWhiteWinter}]\label{thm:GabeTWW}
If $A$ is a separable, exact C*-algebra satisfying the Universal Coefficient Theorem in KK-theory, then every faithful, amenable trace on $A$ is quasidiagonal.
\end{theorem}

We end this introduction with an outline of the proof and of the paper.  Fix a free ultrafilter $\omega$ on $\mathbb{N}$.  Let $\mathcal{Q}_\omega$ and $\mathcal{R}^\omega$ denote the uniform and tracial ultrapowers of the universal UHF-algebra $\mathcal{Q}$ and the hyperfinite $\mathrm{II}_1$-factor $\mathcal{R}$.  Let $\tr_\omega$ and $\tr^\omega$ denote the traces on $\mathcal{Q}_\omega$ and $\mathcal{R}^\omega$ induced by the unique traces on $\mathcal{Q}$ and $\mathcal{R}$.  The following result is well-known to experts.  The second part follows from the proof of Proposition 1.4(ii) in \cite{TikuisisWhiteWinter} and the first part has a similar proof.

\begin{proposition}
Let $\tau$ be a trace on a separable C*-algebra $A$.
\begin{enumerate}
  \item The trace $\tau$ is amenable if, and only if, there is a *-homomorphism $\varphi : A \rightarrow \mathcal{R}^\omega$ with a completely positive, contractive lift $A \rightarrow \ell^\infty(\mathcal{R})$ such that $\mathrm{tr}^\omega \circ \varphi = \tau$.
  \item The trace $\tau$ is quasidiagonal if, and only if, there is a *-homomorphism $\varphi : A \rightarrow \mathcal{Q}_\omega$ with a completely positive, contractive lift $A \rightarrow \ell^\infty(\mathcal{Q})$ such that $\mathrm{tr}_\omega \circ \varphi = \tau$.
\end{enumerate}
\end{proposition}

Viewing $\mathcal{Q}$ as a weak*-dense subalgebra of $\mathcal{R}$, the embedding $\mathcal{Q} \hookrightarrow \mathcal{R}$ induces a surjective, trace-preserving *-homomorphism $\pi: \mathcal{Q}_\omega \rightarrow \mathcal{R}^\omega$ by Theorem 3.3 in \cite{KirchbergRordam}.  This allows us to rephrase Theorem \ref{thm:GabeTWW} as a lifting problem.  In fact, we will deduce Theorem \ref{thm:GabeTWW} from the following lifting result.

\begin{theorem}\label{thm:LiftingTheorem}
If $A$ is a separable, exact C*-algebra satisfying the UCT and $\varphi : A \rightarrow \mathcal{R}^\omega$ is a faithful, nuclear *-homomorphism, then there is a nuclear *-homomorphism $\psi : A \rightarrow \mathcal{Q}_\omega$ with $\pi \circ \psi = \varphi$.  If $A$ and $\varphi$ are unital, we may choose $\psi$ to be unital.
\end{theorem}

We now rephrase the problem once more.  Consider the \emph{trace-kernel ideal} $J := \ker(\pi) \subseteq \mathcal{Q}_\omega$. Given a separable C$^*$-algebra $A$ and a *\=/homomorphism $\varphi : A \rightarrow \mathcal{R}^\omega$, there is a pullback extension
\begin{equation}\label{eqn:PullBack}
 \begin{tikzcd}
   \hphantom{.}\varphi^*\eta: \quad 0 \arrow{r} & J \arrow{r} \arrow[-, double equal sign distance]{d} & E \arrow{r}{\tilde{\pi}} \arrow{d}{\tilde{\varphi}} & A \arrow{r} \arrow{d}{\varphi} & 0\hphantom{.\varphi^*\eta: \quad}\\
   \hphantom{.\varphi^*}\eta: \quad 0 \arrow{r} & J \arrow{r} & \mathcal{Q}_\omega \arrow{r}{\pi} & \mathcal{R}^\omega \arrow{r} & 0.\hphantom{\varphi^*\eta: \quad}
\end{tikzcd}
\end{equation}
To produce a *-homomorphism $\psi : A \rightarrow \mathcal{Q}_\omega$ lifting $\varphi$, it is enough to show $\varphi^*\eta$ admits a *-homomorphic splitting.  When $\varphi$ is nuclear, the extension $\varphi^*\eta$ has a weakly nuclear splitting and hence, ignoring separability issues for the moment,\footnote{The non-separability is handled by first passing to a separable subextension and doing calculations there.  See Section \ref{sec:SeparablePullback} for the details.} the extension defines an element in the Kasparov group $\Ext_{nuc}(A, J)$.  With the help of the Universal Coefficient Theorem, we show this extension defines the trivial element of $\Ext_{nuc}(A, J)$, and therefore, it is stably unitarily equivalent to a split extension.

A version of the Weyl-von Neumann Theorem due to Elliott and Kucerovsky in \cite{ElliottKucerovsky} allows one to control the stabilizations in $\Ext_{nuc}(A, J)$ and to conclude that $\varphi^*\eta$ admits a weakly nuclear, *-homomorphic splitting.  The faithfulness of $\varphi$ is used to verify the hypotheses of this theorem and the exactness of $A$ is used to show the lift $\psi$ obtained from this argument is nuclear.

It worth noting that the original proof of Theorem \ref{thm:GabeTWW} given by Tikuisis, White, and Winter in \cite{TikuisisWhiteWinter} and its refinement by Gabe in \cite{Gabe:TWW} also depend on a version of the Weyl-von Neumann Theorem in a less direct way.  Dadarlat and Eilers prove a version of the Weyl-von Neumann Theorem in
\cite[Theorem 2.22]{DadarlatEilers} which is used to prove a \emph{stable uniqueness theorem} in \cite{DadarlatEilers}.  This stable uniqueness theorem is, in turn, needed at a key step in \cite{TikuisisWhiteWinter} as a way of passing from homotopy equivalence to stable approximate unitary equivalence.  This is also where the faithfulness of $\tau$ and the UCT appear in the original proof of Theorem \ref{thm:GabeTWW} given in \cite{TikuisisWhiteWinter} and \cite{Gabe:TWW}.

The paper is organized as follows.  In Section \ref{sec:SplittingTheorem}, we recall the necessary definitions from the extension theory.  The key technical tool in this paper is Theorem \ref{thm:splitting} which gives a sufficient condition for an extension to split.  In Section \ref{sec:TraceKernelIdeal}, we collect several properties of the trace-kernel ideal $J$ appearing in \eqref{eqn:PullBack} above and show it satisfies all the hypotheses of Theorem \ref{thm:splitting} except separability.  In Section \ref{sec:SeparablePullback}, it is shown that the extension $\eta$ in \eqref{eqn:PullBack} can be replaced with a separable extension so that the corresponding pullback extension satisfies the conditions of our splitting result (Theorem \ref{thm:splitting}).  Finally, Section \ref{sec:MainResult} combines all the pieces and contains the proofs of Theorems \ref{thm:GabeTWW} and \ref{thm:LiftingTheorem}.

\section{A Splitting Result}\label{sec:SplittingTheorem}

The main result of this section is Theorem \ref{thm:splitting} which combines several known results in the extension theory to give a sufficient condition for an extension to split.  Before stating the theorem, we recall some definitions.

A unital C$^*$-algebra $A$ has \emph{stable rank one} if every element in $A$ is a limit of invertible elements in $A$ and has \emph{real rank zero} if every self-adjoint element of $A$ is a limit of invertible, self-adjoint elements in $A$.  A non-unital C$^*$-algebra is said to have stable rank one or real rank zero if its unitization has the corresponding property.

Let $V(A)$ denote the monoid of projections in $\mathbb{M}_\infty(A)$ up to Murray-von Neumann equivalence.  We say $V(A)$ is \emph{almost unperforated} if whenever $x , y \in V(A)$ and $k \geq 1$ is an integer such that $(k+1) x \leq k y$, we have $x \leq y$.

If $A$ and $I$ are C$^*$-algebras, an \emph{extension} of $A$ by $I$ is a short exact sequence
\begin{equation}\label{eqn:Extension}
 \begin{tikzcd} \eta: \quad 0 \arrow{r} & I \arrow{r} & E \arrow{r}{\pi} & A \arrow{r} & 0 \quad \hphantom{\eta:} \end{tikzcd}
\end{equation}
of C$^*$-algebras.  We refer to \cite[Chapter VII]{BlackadarKtheory} for the basic theory of extensions, although we recall some of the relevant definitions below.

The extension $\eta$ is called \emph{semisplit} if there is a completely positive, contractive splitting $\varphi : A \rightarrow E$ (i.e.\ $\pi \circ \varphi = \operatorname{id}_A$) and is called \emph{split} if $\varphi$ can be chosen to be a *-homomorphism.

The extension $\eta$ above induces a *-homomorphism $E \rightarrow M(I)$ where $E$ acts on $I$ by left multiplication.  This induces a *-homomorphism $\beta : A \rightarrow M(I)/I$ called the \emph{Busby invariant} of $\eta$.  The extension is called \emph{full} if for all non-zero $a \in A$, $\beta(a)$ generates $M(I)/I$ as a two-sided ideal.  The extension $\eta$ is called \emph{unitizably full} if the unitized extension
\[ \begin{tikzcd} \eta^\dag: \quad 0 \arrow{r} & J \arrow{r} & E^\dag \arrow{r}{\pi^\dag} & A^\dag \arrow{r} & 0 \quad \hphantom{\eta^\dag:}\end{tikzcd} \]
is full where $A^\dag$ and $E^\dag$ are the unitizations of $A$ and $E$; here, and throughout the paper, if $A$ or $E$ is unital, then a new unit is adjoined.  Note that if $\beta : A \rightarrow M(I)/I$ is the Busby invariant for $\eta$, then the Busby invariant for $\eta^\dag$ is given by $\beta^\dag : A^\dag \rightarrow M(I) / I$, $\beta^\dag(a + \lambda) = \beta(a) + \lambda$ for all $a \in A$ and $\lambda \in \mathbb{C}$.

Given two extensions $\eta_1$ and $\eta_2$ of $A$ by $I$ with Busby invariants $\beta_1, \beta_2 : A \rightarrow M(I) / I$, we say $\eta_1$ and $\eta_2$ are \emph{strongly unitarily equivalent} if there is a unitary $u \in U(M(I))$ such that $\mathrm{ad}(\rho(u)) \circ \beta_1 = \beta_2$ where $\rho : M(I) \rightarrow M(I) / I$ is the quotient map.

Suppose $A$ and $I$ are separable C$^*$-algebras and $I$ is stable.  There are isometries $s_1, s_2 \in M(I)$ such that $s_1s_1^* + s_2s_2^* = 1$.  Given extensions $\eta_1$ and $\eta_2$ of $A$ by $I$ with Busby invariants $\beta_1, \beta_2 : A \rightarrow M(I)$, define a *-homomorphism $\beta : A \rightarrow M(I) / I$ by $\beta(a) = s_1 \beta_1(a) s_1^* + s_2 \beta_2(a)s_2^*$.  Define a C$^*$-algebra
\[ E = \{ (a, b) \in A \oplus M(I) : \beta(a) = \rho(b) \} \]
where $\rho : M(I) \rightarrow M(I) / I$ denotes the quotient map.  If $\pi : E \rightarrow A$ denotes the projection onto the first coordinate and $\iota : I \rightarrow E$ denotes the inclusion into the second coordinate, then there is an extension
\[ \begin{tikzcd} 0 \arrow{r} & I \arrow{r}{\iota} & E \arrow{r}{\pi} & A \arrow{r} & 0 \end{tikzcd}. \]
We denote this extension by $\eta_1 \oplus \eta_2$.  Up to strong unitary equivalence, $\eta_1 \oplus \eta_2$ is independent of the choice of isometries $s_1$ and $s_2$.

When $A$ and $I$ are separable and $I$ is stable, let $\Ext^{-1}(A, I)$ denote the set of equivalence classes of semisplit extensions of $A$ by $I$ where two semisplit extensions $\eta_1$ and $\eta_2$ are equivalent if there is a split extension $\eta$ of $A$ by $I$ such that $\eta_1 \oplus \eta$ and $\eta_2 \oplus \eta$ are strongly unitarily equivalent.  For a semisplit extension $\eta$ of $A$ by $I$, let $\langle \eta \rangle$ denote the element in $\Ext^{-1}(A, I)$ defined by $\eta$.  The set $\Ext^{-1}(A, I)$ is an abelian group with addition given by $\langle \eta_1 \rangle + \langle \eta_2 \rangle = \langle \eta_1 \oplus \eta_2 \rangle$.  The zero element of $\Ext^{-1}(A, I)$ is the equivalence class of a split extension.  Moreover, for a semisplit extension $\eta$ of $A$ by $I$, $\langle \eta \rangle = 0$ if, and only if, there is a split extension $\eta'$ of $A$ by $I$ such that $\eta \oplus \eta'$ is a split extension.

Next we recall the Universal Coefficient Theorem (UCT).  We state this formally for the sake of showing how it appears in the proof of Theorem \ref{thm:GabeTWW}.  Given an extension $\eta$ as in \eqref{eqn:Extension}, there is a six term exact sequence
\[ \begin{tikzcd}
K_0(I) \arrow{r} & K_0(E) \arrow{r} & K_0(A) \arrow{d}{\partial_0} \\
K_1(A) \arrow{u}{\partial_1} & K_1(E) \arrow{l} & K_1(I) \arrow{l}
\end{tikzcd} \]
in K-theory associated to $\eta$.  This construction yields a group homomorphism
\[ \begin{tikzcd} \Ext^{-1}(A, I) \arrow{r}{\alpha} & \Hom_\mathbb{Z}(K_0(A), K_1(I)) \oplus \Hom_\mathbb{Z}(K_1(A), K_0(I)) \end{tikzcd}\]
given by sending $\langle \eta \rangle$ to $\partial_0 \oplus \partial_1$.  If the extension $\eta$ is in the kernel of $\alpha$, then the six-term exact sequence defines an extension
\[ \begin{tikzcd} K_i(\eta): \quad 0 \arrow{r} & K_i(I) \arrow{r} & K_i(E) \arrow{r} & K_i(A) \arrow{r} & 0 \quad \hphantom{K_i(\eta):} \end{tikzcd} \]
of abelian groups for $i = 0, 1$.  This yields a group homomorphism
\[ \begin{tikzcd} \ker(\alpha) \arrow{r}{\gamma} & \Ext_\mathbb{Z}^1(K_0(A), K_0(I)) \oplus \Ext_\mathbb{Z}^1(K_1(A), K_1(I)) \end{tikzcd} \]
defined by sending $\langle \eta \rangle$ to $K_0(\eta) \oplus K_1(\eta)$.  We say a separable C$^*$-algebra $A$ \emph{satisfies the UCT} if for all separable C$^*$-algebras $J$, the map $\alpha$ defined above is surjective and the map $\gamma$ defined above is an isomorphism.  In this case, there is a short exact sequence
\[ \begin{tikzcd} \Ext_\mathbb{Z}^1(K_*(A), K_*(I)) \arrow[tail]{r}{\gamma^{-1}} & \Ext^{-1}(A, J) \arrow[two heads]{r}{\alpha} & \Hom_\mathbb{Z}(K_*(A), K_{*+1}(I)). \end{tikzcd}\]

When the algebra $A$ is not assumed to be nuclear, a slight refinement of $\Ext^{-1}(A, I)$ is often more appropriate.  

A completely positive splitting $\sigma : A \rightarrow E$ for an extension $\eta$ as in \eqref{eqn:Extension} is called \emph{weakly nuclear} if for all $x \in I$, the map $A \rightarrow I$ given by $a \mapsto x \sigma(a) x^*$ is nuclear.  We say an extension is \emph{weakly nuclear} if it admits a weakly nuclear splitting.  By requiring all extensions in the definition of $\Ext^{-1}(A, I)$ to be weakly nuclear, one obtains an abelian group $\Ext_{\mathrm{nuc}}(A, I)$.  If either $A$ or $I$ is KK-equivalent to a nuclear C$^*$-algebra, then the canonical map
\[ \begin{tikzcd} \Ext_{\mathrm{nuc}}(A, I) \arrow{r} & \Ext^{-1}(A, I) \end{tikzcd} \]
is an isomorphism (see \cite{Skandalis:KKnuc}).  In particular, this is the case if either $A$ or $I$ is nuclear or if either $A$ or $I$ satisfies the UCT (Theorem 23.10.5 in \cite{BlackadarKtheory}).

Before stating the key splitting theorem, we make the following technical definition.  First recall that an abelian group $G$ is called \emph{divisible} if for every $g \in G$ and every $n \in \mathbb{Z} \setminus \{0\}$, there is an $h \in G$ such that $nh = g$.

\begin{definition}
A C$^*$-algebra $I$ is called an \emph{admissible kernel} if $I$ has real rank zero, stable rank one, $V(I)$ is almost unperforated, $K_0(I)$ is divisible, $K_1(I) = 0$, and every projection in $I \otimes \mathcal{K}$ is Murray-von Neumann equivalent to a projection in $I$.
\end{definition}

In Section \ref{sec:TraceKernelIdeal}, we will show that the trace-kernel ideal $J$ appearing in the diagram \eqref{eqn:PullBack} is an admissible kernel.  In Section \ref{sec:SeparablePullback}, we will further show how to produce arbitrarily large separable subalgebras of $J$ which are also admissible kernels.

The following theorem is the main technical tool in this paper.  The key result is a version of the Weyl-von Neumann Theorem proven by Elliott and Kucerovsky in \cite{ElliottKucerovsky} (with a correction in the non-unital case given by Gabe in \cite{Gabe:AbsorbingExtensions}) which characterizes nuclearly absorbing extensions.  An extension $\eta$ of a C$^*$-algebra $A$ by a stable C$^*$-algebra $I$ is called \emph{nuclearly absorbing} if for every extension $\eta'$ of $A$ by $I$ with a weakly nuclear, *-homomorphic splitting, the extension $\eta \oplus \eta'$ is strongly unitarily equivalent to $\eta$.

\begin{theorem}\label{thm:splitting}
Suppose $A$ and $I$ are separable C*-algebras and
\[ \begin{tikzcd} \eta: \quad 0 \arrow{r} & I \arrow{r} & E \arrow{r}{\pi} & A \arrow{r} & 0 \quad \hphantom{\eta:} \end{tikzcd} \]
is a weakly nuclear, unitizably full extension such that $I$ is an admissible kernel and the index map $K_1(A) \rightarrow K_0(I)$ is trivial.  If $A$ satisfies the UCT, then $\eta$ has a weakly nuclear, *-homomorphic splitting $A \rightarrow E$.
\end{theorem}

\begin{proof}
As $I$ has stable rank one, $I$ has cancellation of projections by Proposition 6.5.1 in \cite{BlackadarKtheory}.  As $I$ is separable and has real rank zero, $I$ has a countable approximate unit consisting of projections by the equivalence of (1) and (3) in Theorem 6.5.6 of \cite{BlackadarKtheory}.  Now, since every projection in $I \otimes \mathcal{K}$ is Murray-von Neumann equivalent to a projection in $I$, the equivalence of (1) and (2) in Proposition 3.4 in \cite{Rordam:Stability} implies $I$ is stable.

Since $I$ is separable, has real rank zero, and $V(I)$ is almost unperforated, $I$ has the corona factorization property by Corollary 5.9 of \cite{OrtegoPereraRordam}.  As $\eta$ is unitizably full, $I$ is stable, and $I$ has the corona factorization property, $\eta$ is nuclearly absorbing by Theorem 2.6 in \cite{Gabe:AbsorbingExtensions}.

Since $\eta$ is weakly nuclear, $\eta$ defines an element $\langle \eta \rangle$ in $\Ext_\mathrm{nuc}(A, I)$.  We claim $\langle \eta \rangle = 0$.  As $A$ satisfies the UCT, $A$ is KK-equivalent to a nuclear C$^*$-algebra and hence the canonical map
\[ \begin{tikzcd} \Ext_{\mathrm{nuc}}(A, I) \arrow{r} & \Ext^{-1}(A, I) \end{tikzcd} \]
is an isomorphism.  Since $K_1(I) = 0$ and $K_0(I)$ is divisible, $\Ext_\mathbb{Z}^1(K_*(A), K_*(I)) = 0$ as divisible abelian groups are injective by Lemma 4.2 of Chapter XX in \cite{Lang:Algebra}.  Hence the UCT yields an isomorphism
\[ \begin{tikzcd} \Ext_{\mathrm{nuc}}(A, I) \arrow{r}{\cong} & \Hom_\mathbb{Z}(K_1(A), K_0(I)) \end{tikzcd} \]
sending the extension $\langle \eta \rangle$ to the index map $K_1(A) \rightarrow K_0(I)$ which is assumed to be the zero map.  So $\langle \eta \rangle = 0$.

Now, there is an extension $\eta'$ of $A$ by $I$ with a weakly nuclear, *-homomorphic splitting such that $\eta \oplus \eta'$ admits a weakly nuclear, *-homomorphic splitting.  Since $\eta$ is nuclearly absorbing, $\eta$ and $\eta \oplus \eta'$ are strongly unitarily equivalent.  The result follows.
\end{proof}

\section{The Trace-Kernel Ideal is Admissible}\label{sec:TraceKernelIdeal}

Let $\mathcal{Q}$ denote the universal UHF-algebra and let $\mathcal{R}$ denote the hyperfinite $\mathrm{II}_1$-factor.  Fix a free ultrapower $\omega$ on the natural numbers, let $\mathcal{Q}_\omega$ denote the norm ultrapower of $\mathcal{Q}$, and let $\mathcal{R}^\omega$ denote the tracial ultrapower of $\mathcal{R}$.  There is a unital inclusion $\mathcal{Q} \hookrightarrow \mathcal{R}$ with weak*-dense range.  By Theorem 3.3 in \cite{KirchbergRordam}, this inclusion induces a surjection $\pi : \mathcal{Q}_\omega \rightarrow \mathcal{R}^\omega$.  The ideal $J := \ker(\pi)$ of $\mathcal{Q}_\omega$ is called the \emph{trace-kernel ideal}.

\begin{theorem}\label{thm:TraceKernelIdeal}
The trace-kernel ideal $J$ is an admissible kernel.
\end{theorem}

We will prove this result is several steps.  The rest of this section is devoted to the proof.

\begin{proposition}
The C*-algebras $J$ and $\mathcal{Q}_\omega$ have real rank zero and stable rank one.
\end{proposition}

\begin{proof}
As both properties are preserved by taking ideals (see Theorem 4.3 in \cite{Rieffel:StableRank} and Corollary 2.8 in \cite{BrownPedersen:RealRankZero}), it is enough to show $\mathcal{Q}_\omega$ has real rank zero and stable rank one.  Since $\mathcal{Q}$ has stable rank one, the proof of Lemma 19.2.2(i) in \cite{Loring:book} shows every element $x \in \mathcal{Q}_\omega$ has a unitary polar decomposition $u|x|$, and then $x$ is the limit of the invertible elements $u(|x| + \varepsilon)$ as $\varepsilon \rightarrow 0^+$.  When $x$ is self-adjoint, using that $\mathcal{Q}$ has real rank zero, the unitary $u$ constructed in the proof can be chosen to be self-adjoint and to commute with $|x|$ so that $u(|x| + \varepsilon)$ is invertible and self-adjoint for all $\varepsilon > 0$.
\end{proof}

We will now compute the $K$-theory of $J$.  Let $\ell^\infty(\mathbb{Q})$ denote the abelian group consisting of bounded sequences of rational numbers and define
\[ \mathbb{Q}_\omega = \ell^\infty(\mathbb{Q}) / \{ q \in \ell^\infty(\mathbb{Q}) : \{ n \in \mathbb{N} : q_n = 0 \} \in \omega \}. \]
The trace $\tr$ on $\mathcal{Q}$ induces an isomorphism $K_0(\mathcal{Q}_\omega) \rightarrow \mathbb{Q}_\omega$ as in Proposition 2.6 of \cite{RainoneSchafhauser} which is an ultraproduct formulation of (a special case of) the results in Section 3.2 of \cite{DadarlatEilers}. Let $\lim_\omega$ denote the map $\mathbb{Q}_\omega \rightarrow \mathbb{R}$ given by taking the limit of a sequence in $\ell^\infty(\mathbb{Q})$ along $\omega$ in $\mathbb{R}$.  Let $G_0$ denote the kernel of this map and note that
\[ G_0 \cong \{ q \in \ell^\infty(\mathbb{Q}) : \lim_{n \rightarrow \omega} q_n = 0 \} / \{ q \in \ell^\infty(\mathbb{Q}) : \{n \in \mathbb{N} : q_n = 0 \} \in \omega \}. \]
We let $\mathbb{Q}_\omega^+ \subseteq \mathbb{Q}_\omega$ denote the monoid consisting of elements of $\mathbb{Q}_\omega$ represented by bounded sequences in $\mathbb{Q}^+$ and let $G_0^+ = \mathbb{Q}_\omega^+ \cap G_0$.  Note that $G_0$ is divisible and $G_0^+$ is almost unperforated.

\begin{proposition}\label{prop:KTheoryComputation}
The algebras $J$, $\mathcal{Q}_\omega$, and $\mathcal{R}^\omega$ have trivial $K_1$-group and there is an isomorphism of extensions
\begin{equation}\label{eqn:KTheoryComputation}
 \begin{tikzcd}
   \hphantom{.}0 \arrow{r} & K_0(J) \arrow{r}{K_0(\iota)} \arrow{d}{K_0(\tr)_\omega} & K_0(\mathcal{Q}_\omega) \arrow{r}{K_0(\pi)} \arrow{d}{K_0(\tr)_\omega} & K_0(\mathcal{R}^\omega) \arrow{r} \arrow{d}{K_0(\tr^\omega)} & 0\hphantom{.} \\
   \hphantom{.}0 \arrow{r} & G_0 \arrow{r} & \mathbb{Q}_\omega \arrow{r}{\lim_{\omega}} & \mathbb{R} \arrow{r} & 0.
\end{tikzcd}
\end{equation}
Moreover, the vertical maps restrict to monoid isomorphisms $\theta_J : V(J) \rightarrow G_0^+$, $\theta_{\mathcal{Q}_\omega} : V(\mathcal{Q}_\omega) \rightarrow \mathbb{Q}_\omega^+$, and $\theta_{\mathcal{R}^\omega} : V(\mathcal{R}^\omega) \rightarrow \mathbb{R}^+$.
\end{proposition}

\begin{proof}
As noted above, $K_0(\tr)_\omega$ is an isomorphism.  Since $\mathcal{R}^\omega$ is a $\mathrm{II}_1$-factor, we have $K_0(\tr^\omega)$ is an isomorphism and $\mathcal{R}^\omega$ has trivial $K_1$-group.

A standard argument also shows $\mathcal{Q}_\omega$ has trivial $K_1$-group.  Indeed, given a unitary $u \in \mathbb{M}_k(\mathcal{Q}_\omega)$, there is an element $(x_n)_n \in \mathbb{M}_k(\ell^\infty(\mathcal{Q}))$ lifting $u$.  Then $\|x_n x_n^* - 1\| \rightarrow 0$ and $\|x_n^* x_n - 1\| \rightarrow 0$ along $\omega$.  Hence by the polar decomposition, there are unitaries $u_n \in \mathbb{M}_k(\mathcal{Q})$ such that $\|x_n - u_n \| \rightarrow 0$ along $\omega$.  Applying a similar perturbation argument, we may assume each $u_n$ belongs to a finite-dimensional subalgebra of $\mathbb{M}_k(\mathcal{Q})$, and hence we may write $u_n = e^{2\pi i h_n}$ for some self-adjoint $h_n \in \mathbb{M}_k(\mathcal{Q})$ with $0 \leq h_n \leq 1$.  The sequence $(h_n)_{n=1}^\infty$ defines a self-adjoint $h \in \mathbb{M}_k(\mathcal{Q}_\omega)$ with $u = e^{2\pi i h}$, and hence $u$ is homotopic to the identity in $\mathbb{M}_k(\mathcal{Q}_\omega)$.  This shows $K_1(\mathcal{Q}_\omega) = 0$.

Since $\lim_\omega$ is surjective, $K_0(\pi)$ is surjective.  As $K_0(\pi)$ is surjective and $K_1(\mathcal{Q}_\omega) = 0$, the six-term exact sequence in $K$-theory implies $K_1(J) =0$.  Now the five lemma implies $K_0(\tr)_\omega : K_0(J) \rightarrow G_0$ is an isomorphism.

Since $J$, $\mathcal{Q}_\omega$, and $\mathcal{R}^\omega$ have stable rank one, the injectivity of the vertical maps in \eqref{eqn:KTheoryComputation} imply the injectivity of $\theta_J$, $\theta_{\mathcal{Q}_\omega}$, and $\theta_{\mathcal{R}^\omega}$.  Since $\mathcal{R}^\omega$ is a $\mathrm{II}_1$-factor, for any $t \in \mathbb{R}^+$, there is a projection $p \in \mathbb{M}_\infty(\mathcal{R}^\omega)$ with $\mathrm{tr}^\omega(p) = t$, and hence $\theta_{\mathcal{R}^\omega}$ is surjective.

Given an element $t \in \mathbb{Q}_\omega^+$ represented by a bounded sequence $(t_n)_{n=1}^\infty \subseteq \mathbb{Q}^+$, let $d$ be an integer with $0 \leq t_n \leq d$ for all $n \geq 1$.  Then there is a sequence of projections $(p_n)_{n=1}^\infty \subseteq \mathbb{M}_d(\mathcal{Q})$ such that $\mathrm{tr}(p_n) = t_n$.  The sequence $(p_n)_{n=1}^\infty$ defines a projection $p \in \mathbb{M}_d(\mathcal{Q}_\omega)$ with $\theta_{\mathcal{Q}_\omega}([p]) = t$, so $\theta_{\mathcal{Q}_\omega}$ is surjective.

Now suppose $t \in G_0^+$.  There is an integer $d \geq 1$ and a projection $p \subseteq \mathbb{M}_d(\mathcal{Q}_\omega)$ with $\theta_{\mathcal{Q}_\omega}([p]) = t$.  Let $(p_n)_{n=1}^\infty \subseteq \mathbb{M}_d(\mathcal{Q})$ be a sequence of projections representing $p$.  As $t \in G_0$, we have $\lim_{n \rightarrow \omega} \mathrm{tr}(p_n) = \lim_{n \rightarrow \omega} t_n = 0$, and hence $p \in \mathbb{M}_d(J)$.  Also, as $\theta_{\mathcal{Q}_\omega}([p]) = t$, we have $\theta_J([p]) = t$, so $\theta_J$ is surjective.
\end{proof}

\begin{corollary}
The group $K_0(J)$ is divisible and $V(J)$ is almost unperforated.
\end{corollary}

\begin{proof}
By Proposition \ref{prop:KTheoryComputation}, $K_0(J) \cong G_0$ and $V(J) \cong G_0^+$.  Since $G_0$ is divisible and $G_0^+$ is almost unperforated, the result follows.
\end{proof}

\begin{proposition}
Every projection in $\mathbb{M}_\infty(J)$ is Murray-von Neumann equivalent to a projection in $J$.
\end{proposition}

\begin{proof}
Fix an integer $d \geq 1$ and a projection $p \in \mathbb{M}_d(J)$, and let $(p_n)_{n=1}^\infty \subseteq \mathbb{M}_d(\mathcal{Q})$ be a sequence of projections representing $p$.  As $p \in J$, we have $\lim_{n \rightarrow \omega} \mathrm{tr}(p_n) = 0$.  Now, the set $S = \{ n \in \mathbb{N} : 0 \leq \mathrm{tr}(p_n) \leq 1 \}$ is in $\omega$.  For $n \in S$, let $q_n \in \mathcal{Q}$ be a projection with $\mathrm{tr}(q_n) = \mathrm{tr}(p_n)$, and for $n \notin S$, let $q_n = 0$.  Then $(q_n)_{n=1}^\infty$ defines a projection in $q \in \mathcal{Q}_\omega$.  As $p \in \mathbb{M}_d(J)$, we have
\[ \lim_{n \rightarrow \omega} \mathrm{tr}(q_n) = \lim_{n \rightarrow \omega} \mathrm{tr}(p_n) = 0, \]
and hence $q \in J$.  Moreover, $\theta_J([p]) = \theta_J([q])$ in $G_0^+$, and hence $[p] = [q]$ as $\theta_J$ is injective by Proposition \ref{prop:KTheoryComputation}.  So $p$ and $q$ are Murray von-Neumann equivalent.
\end{proof}

This completes the proof of Theorem \ref{thm:TraceKernelIdeal}.

\section{A Separable Pullback Extension}\label{sec:SeparablePullback}

In this section, we will replace the extension
\[ \begin{tikzcd} \eta: \quad 0 \arrow{r} & J \arrow{r} & \mathcal{Q}_\omega \arrow{r} & \mathcal{R}^\omega \arrow{r} & 0 \hphantom{\eta: \quad} \end{tikzcd} \]
given in \eqref{eqn:PullBack} above with a separable extension and produce an extension of $A$ by a separable subalgebra of $J$ satisfying the conditions of Theorem \ref{thm:splitting}.

The following three results are a standard applications of Blackadar's notion of separable inheritability (see Section II.8.5 in \cite{BlackadarEncyclopedia}).  A property $\mathcal{P}$ of C$^*$-algebras is called \emph{separably inheritable} if it is preserved by sequential direct limits with injective connecting maps and for every C$^*$-algebra $A$ with property $\mathcal{P}$ and every separable subalgebra $A_0 \subseteq A$, there is a separable subalgebra $A_1 \subseteq A$ containing $A_0$ such that $A_1$ has property $\mathcal{P}$.

\begin{proposition}\label{prop:SeparableKernel}
The property of being an admissible kernel is separably inheritable.
\end{proposition}

\begin{proof}
The properties $K_1(\cdot) = 0$, real rank zero, and stable rank one are separably inheritable by Paragraph II.8.5.5 in \cite{BlackadarEncyclopedia}.  Very similar proofs show the properties $V(\cdot)$ is almost unperforated, $K_0(\cdot)$ is divisible, and every projection in the stabilization of the algebra being Murray von-Neumann equivalent to a projection in the algebra are separably inheritable.  The result follows from Proposition II.8.5.3 in \cite{BlackadarEncyclopedia} which states that the meet of countably many separably inheritable properties is separably inheritable.
\end{proof}

\begin{proposition}\label{prop:SeparableExtension}
Consider an extension
\[ \begin{tikzcd} 0 \arrow{r} & I \arrow{r} & B \arrow{r}{\pi} & D \arrow{r} & 0 \end{tikzcd} \]
of C*-algebra such that $I$ is an admissible kernel and $B$ and $D$ are unital.  Given a separable, unital subalgebra $D_0 \subseteq D$, there is a separable, unital subalgebra $B_0 \subseteq B$ such that $\pi(B_0) = D_0$ and $B_0 \cap I$ is an admissible kernel.
\end{proposition}

\begin{proof}
Fix a countable dense set $T \subseteq D_0$ and let $S \subseteq B$ be a countable set such that $\pi(S) = T$.  Let $B_1$ denote the unital subalgebra of $B$ generated by $S$ and note that $\pi(B_1) = D_0$ since *-homomorphisms have closed range.  By Proposition \ref{prop:SeparableKernel}, there is a non-zero, separable admissible kernel $I_1 \subseteq I$ such that $B_1 \cap I \subseteq I_1$.  Let $B_2$ denote the subalgebra of $B$ generated by $B_1$ and $I_1$ and choose again a separable admissible kernel $I_2 \subseteq I$ with $B_2 \cap I \subseteq I_2$.  Continuing inductively, we have increasing an sequence $(I_n)_n$ of separable admissible kernels contained in $I$ and an increasing sequence $(B_n)_n$ of separable, unital subalgebras of $B$ such that $\pi(B_n) = D_0$ and $B_n \cap I \subseteq I_n \subseteq B_{n+1} \cap I$ for all $n \geq 1$.

Let $B_0$ denote the closed union of the $B_n$ and note that $\pi(B_0) = D_0$.  Moreover, $I \cap B_0$ is the closed union of the algebras $I_n$ and since each $I_n$ is an admissible kernel, so is $I \cap B_0$ by Proposition \ref{prop:SeparableKernel}.
\end{proof}

\begin{proposition}\label{prop:SeparableQuotient}
If $A$ is a separable C*-algebra and $\varphi : A \rightarrow \mathcal{R}^\omega$ is a nuclear *-homomorphism, then there is a separable, unital C*-subalgebra $R_0 \subseteq \mathcal{R}^\omega$ such that $R_0$ is simple, $K_1(R_0) = 0$, $\varphi(A) \subseteq R_0$, and $\varphi$ is nuclear as a map $A \rightarrow R_0$.
\end{proposition}

\begin{proof}
As $\varphi$ is nuclear, there are completely positive, contractive maps $\theta_n : A \rightarrow \mathbb{M}_{k(n)}$ and $\rho_n : \mathbb{M}_{k(n)} \rightarrow \mathcal{R}^\omega$ such that $\rho_n \circ \theta_n : A \rightarrow \mathcal{R}^\omega$ converge point norm to $\varphi$.  Let $R_1$ denote the unital C$^*$-subalgebra of $\mathcal{R}^\omega$ generated by the subspaces $\rho_n(\mathbb{M}_{k(n)})$ and note that $R_1$ is separable.  By construction, $\varphi(A) \subseteq R_1$ and $\varphi$ is nuclear as a map $A \rightarrow R_1$.

As $\mathcal{R}^\omega$ is $\mathrm{II}_1$-factor, $\mathcal{R}^\omega$ is simple and $K_1(\mathcal{R}^\omega) =0$.  Since simplicity and $K_1(\cdot) = 0$ are separably inheritable properties (see Paragraph II.8.5.5 and Theorem II.8.5.6 in \cite{BlackadarEncyclopedia}), there is a separable, simple, unital C$^*$-subalgebra $R_0 \subseteq \mathcal{R}^\omega$ containing $R_1$ such that $K_0(R_0) = 0$.  As $\varphi$ is nuclear as a map $A \rightarrow R_1$ and $R_1 \subseteq R_0$, $\varphi$ is nuclear as a map $A \rightarrow R_0$.
\end{proof}

Fix a separable C$^*$-algebra $A$ and a faithful, nuclear *-homomorphism $\varphi : A \rightarrow \mathcal{R}^\omega$.  Choose $R_0$ as in Proposition \ref{prop:SeparableQuotient}.  By Proposition \ref{prop:SeparableExtension} and Theorem \ref{thm:TraceKernelIdeal}, there is a separable, unital subalgebra $Q_0 \subseteq \mathcal{Q}_\omega$ such that $J_0 := Q_0 \cap J$ is an admissible kernel and $\pi(Q_0) = R_0$.  Consider the pullback diagram
\begin{equation}\label{eqn:SeparablePullback}
 \begin{tikzcd}
\varphi_0^*\eta_0: \quad 0 \arrow{r} & J_0 \arrow{r}{\tilde{\iota}_0} \arrow[-, double equal sign distance]{d} & E_0 \arrow{r}{\tilde{\pi}_0} \arrow{d}{\tilde{\varphi}_0} & A \arrow{r} \arrow{d}{\varphi_0} & 0 \hphantom{\varphi_0^*\eta_0: \quad} \\
\hphantom{\varphi^*}\eta_0: \quad 0 \arrow{r} & J_0 \arrow{r}{\iota_0} & Q_0 \arrow{r}{\pi_0} & R_0 \arrow{r} & 0 \hphantom{\varphi_0^*\eta_0: \quad}
\end{tikzcd}
\end{equation}
where $\pi_0$ is the restriction of $\pi : \mathcal{Q}_\omega \rightarrow \mathcal{R}^\omega$ to $Q_0$ and $\varphi_0$ denotes the morphism $\varphi : A \rightarrow \mathcal{R}^\omega$ viewed as a map into $R_0$.  Explicitly,
\[ E_0 := \{ (a, q) \in A \oplus Q_0 : \varphi_0(a) = \pi_0(q) \}, \]
$\tilde{\pi}_0$ and $\tilde{\varphi}_0$ are the projection maps on the first and second coordinates, respectively, $\iota_0$ is the inclusion, and $\tilde{\iota}_0$ is the inclusion into the second component of $E_0$.

\begin{theorem}\label{thm:PullbackExtension}
With the notation above, assume either $A$ is non-unital or $A$ is unital but $\varphi(1) \neq 1$.  Then the extension $\varphi_0^*\eta_0$ in \eqref{eqn:SeparablePullback} satisfies the hypothesis of Theorem \ref{thm:splitting}.  Hence if $A$ satisfies the UCT, then $\varphi_0^*\eta_0$ admits a weakly nuclear, *-homomorphic splitting.
\end{theorem}

\begin{proof}
By the naturality of the index map, the index map $K_1(A) \rightarrow K_0(J_0)$ associated to the extension $\varphi_0^*\eta_0$ factors through $K_1(R_0) = 0$ and hence is the zero map.

To see the extension is unitizably full, define $\varphi_0^\dag : A^\dag \rightarrow \mathcal{R}^\omega$ by $\varphi_0^\dag(a + \lambda 1) = \varphi_0(a) + \lambda 1$ and note that $\varphi_0^\dag$ is faithful as $\varphi$ (and hence $\varphi_0$) is faithful and non-unital. If $\beta : R_0 \rightarrow M(J_0) /J_0$ denotes the Busby invariant of $\eta_0$ and $\tilde{\beta}^\dag : A^\dag \rightarrow M(J_0)/J_0$ denotes the Busby invariant of the unitization of $\varphi_0^*\eta_0$, then $\tilde{\beta}^\dag = \beta \circ \varphi_0^\dag$.  Since $\varphi_0^\dag$ is faithful, it suffices to show $\beta$ is full, but this is immediate since $R_0$ is simple and $\beta$ is unital.

Finally, we show $\varphi_0^*\eta_0$ admits a weakly nuclear splitting $\sigma : A \rightarrow E_0$.  As $\varphi_0$ is nuclear, the Choi-Effros Lifting Theorem implies there is a nuclear map $\tilde{\varphi}_0 : A \rightarrow Q_0$ such that $\pi_0 \circ \tilde{\varphi}_0 = \varphi_0$. Define $\sigma : A \rightarrow E$ by $\sigma(a) = (a, \tilde{\varphi}_0(a))$ and note that $\tilde{\pi}_0(\sigma(a)) = a$ for all $a \in A$.  Given $x \in J_0$, we have $\tilde{\iota}_0(x) \sigma(a) \tilde{\iota}_0(x)^* = (0, x \tilde{\varphi}(a) x^*)$ for all $a \in A$.  As $\tilde{\varphi}_0$ is nuclear, so is $\tilde{\iota}_0(x) \sigma(\cdot) \tilde{\iota}_0(x)^*$ for all $x \in J_0$ and hence $\varphi_0^*\eta_0$ is weakly nuclear.
\end{proof}

\section{The Proofs of Theorems \ref{thm:GabeTWW} and \ref{thm:LiftingTheorem}}\label{sec:MainResult}

\begin{proof}[\textbf{Proof of Theorem \ref{thm:LiftingTheorem}}]
Assume first either $A$ is not unital or $A$ is unital and $\varphi(1) \neq 1$.  Consider the pullback extension $\varphi_0^*\eta_0$ in \eqref{eqn:SeparablePullback} above.  By Theorem \ref{thm:PullbackExtension}, the extension $\varphi_0^*\eta_0$ admits a weakly nuclear *-homomorphic splitting $\sigma : A \rightarrow E_0$.  Define $\psi$ as the composition
\[ \begin{tikzcd} A \arrow{r}{\sigma} & E_0 \arrow{r}{\tilde{\varphi}_0} & Q_0 \arrow[hook]{r} & \mathcal{Q}_\omega \end{tikzcd} \]
and note that $\pi \circ \psi = \varphi$.

We now work to show the $\psi$ is nuclear; in fact, we will show $\psi_0 := \tilde{\varphi}_0 \circ \sigma$ is nuclear.  It suffices to show that for any C$^*$-algebra $D$, the map
\[ \psi_0 \otimes \mathrm{id}_D : A \otimes_\mathrm{max} D \rightarrow Q_0 \otimes_\mathrm{max} D \]
factors through $A \otimes_\mathrm{min} D$ by Corollary 3.8.8 in \cite{BrownOzawa}.  Let $x$ be in the kernel of the canonical quotient map $A \otimes_\mathrm{max} D \rightarrow A \otimes_\mathrm{min} D$ and let $y = (\psi_0 \otimes \mathrm{id}_D)(x)$.  It suffices to show $y = 0$.

Let $\lambda_0 : Q_0 \rightarrow M(J_0)$ denote the *-homomorphism given by left multiplication.  Then $\lambda_0 \circ \psi_0$ is weakly nuclear and hence is nuclear by Proposition 3.2 in \cite{Gabe:LiftingTheorems} since $A$ is exact.  Therefore,
\[ (\lambda_0 \otimes \mathrm{id}_D)(y) = (\lambda_0 \circ \psi_0 \otimes \mathrm{id}_D)(x) = 0 \]
by Corollary 3.8.8 in \cite{BrownOzawa}.  As $\pi_0 \circ \psi_0 = \varphi_0$ is nuclear, we also have
\[ (\pi_0 \otimes \mathrm{id}_D)(y) = (\pi_0 \circ \psi_0 \otimes \mathrm{id}_D)(x) = 0 \]
by Corollary 3.8.8 in \cite{BrownOzawa}.  As maximal tensor products preserve exact sequences, there is a commuting diagram
\[ \begin{tikzcd}
0 \arrow{r} & J_0 \otimes_\mathrm{max} D \arrow{r}{\iota_0 \otimes \mathrm{id}_D} \arrow[equals]{d} & Q_0 \otimes_\mathrm{max} D \arrow{r}{\pi_0 \otimes \mathrm{id}_D} \arrow{d}{\lambda_0 \otimes \mathrm{id}_D} & R_0 \otimes_\mathrm{max} D \arrow{r} \arrow{d} & 0 \\
0 \arrow{r} & J_0 \otimes_\mathrm{max} D \arrow{r} & M(J_0) \otimes_\mathrm{max} D \arrow{r} & M(J_0) / J_0 \otimes_\mathrm{max} D \arrow{r} & 0
\end{tikzcd} \]
with exact rows.  Since $(\pi_0 \otimes \mathrm{id}_D)(y) = 0$ and $(\lambda_0 \otimes \mathrm{id}_D)(y) = 0$, we have $y = 0$ as claimed.  Hence $\psi$ is nuclear which completes the proof when either $A$ is non-unital or $\varphi$ does not preserve the unit.

Now assume $A$ and $\varphi$ are unital.  We first show there is a unital, nuclear *-homomorphism $\psi' : A \rightarrow \mathcal{Q}_\omega$ with $\tr_\omega \circ \psi' = \tr^\omega \circ \varphi$.  Let $\varphi_1$ denote the composition
\[ \begin{tikzcd} A \arrow{r}{\varphi} & \mathcal{R}^\omega \arrow[hook]{r} & \mathbb{M}_2(\mathcal{R}^\omega) \arrow{r}{\cong} & (\mathbb{M}_2(\mathcal{R}))^\omega \arrow{r}{\cong} & \mathcal{R}^\omega \end{tikzcd} \]
where the second map is the (non-unital) embedding into the upper left corner and note that $\varphi_1$ is faithful, nuclear, and non-unital.  Hence by the first part of the proof, there is a nuclear *-homomorphism $\psi_1': A \rightarrow \mathcal{Q}_\omega$ such that $\pi \circ \psi_1' = \varphi_1$.  Note that
\[ \mathrm{tr}_\omega \circ \psi_1' = \mathrm{tr}^\omega \circ \pi \circ \psi_1' = \mathrm{tr}^\omega \circ \varphi_1 = \frac12 \mathrm{tr}^\omega \circ \varphi. \]
By the equivalence of (b) and (c) in Proposition 3.4(ii) of \cite{Gabe:TWW} applied to the trace $\mathrm{tr}^\omega \circ \varphi$ on $A$, there is a unital, nuclear *-homomorphism $\psi' : A \rightarrow \mathcal{Q}_\omega$ such that $\mathrm{tr}_\omega \circ \psi' = \mathrm{tr}^\omega \circ \varphi$.

Let $\tau = \mathrm{tr}^\omega \circ \varphi$.  Since $\varphi$ is nuclear, $\varphi$ has a completely positive lift $A \rightarrow \ell^\infty(\mathcal{R})$ by the Choi-Effros Lifting Theorem and hence $\tau$ is amenable.  As $A$ is exact, $A$ is locally reflexive (see (11) of Section 4 in \cite{Kirchberg:SubalgebrasOfCAR}), and hence $\tau$ is uniformly amenable by Theorem 4.3.3 in \cite{BrownQDTraces}.  Now by the equivalence of (1) and (5) of Theorem 3.3.2 in \cite{BrownOzawa}, $\pi_\tau(A)''$ is hyperfinite, where $\pi_\tau$ denotes the GNS representation of $A$ corresponding to the tracial state $\tau$.

Note that the morphisms $\varphi, \pi \circ \psi' : A \rightarrow \mathcal{R}^\omega$ extend to normal *-homomorphisms $\bar{\varphi}, \bar{\psi}' : \pi_\tau(A)'' \rightarrow \mathcal{R}^\omega$ and $\mathrm{tr}^\omega \circ \bar{\varphi} = \mathrm{tr}^\omega \circ \bar{\psi}'$.  By the classification of normal *-homomorphisms from separable hyperfinite von Neumann algebras to $\mathrm{II}_1$-factors (see the proof of Proposition 2.1 in \cite{CuipercaGiordanoNgNiu}), $\bar{\varphi}$ and $\bar{\psi}'$ are approximately unitarily equivalent.  By a reindexing argument, there is a unitary $u \in \mathcal{R}^\omega$ such that $\bar{\varphi} = \mathrm{ad}(u) \circ \bar{\psi}'$.  Now, $\varphi = \mathrm{ad}(u) \circ \pi \circ \psi'$.  As the unitary group of $\mathcal{R}^\omega$ is path connected, there is a unitary $v \in \mathcal{Q}_\omega$ such that $\pi(v) = u$.  Define $\psi = \mathrm{ad}(v) \circ \psi' : A \rightarrow \mathcal{Q}_\omega$ and note that $\pi \circ \psi = \varphi$.  Since $\psi'$ is nuclear, $\psi$ is also nuclear.
\end{proof}

In the setting of Theorem \ref{thm:GabeTWW}, the nuclearity condition on $\varphi$ in Theorem \ref{thm:LiftingTheorem} is automatic.  The following lemma is implicitly contained in the proof of Proposition 3.5 in \cite{Gabe:TWW} with the key tool being Lemma 3.3 in \cite{Dadarlat:QDMorphisms}.

\begin{lemma}\label{lem:Exactness}
Suppose $A$ is an exact C*-algebra and $\varphi : A \rightarrow \mathcal{R}^\omega$ is a completely positive, contractive map with a completely positive, contractive lift $\tilde{\varphi} : A \rightarrow \ell^\infty(\mathcal{R})$.  Then $\varphi$ is nuclear.
\end{lemma}

\begin{proof}
Let $\varphi_n : A \rightarrow \mathcal{R}$ denote the composition of $\tilde{\varphi}$ with the projection $\ell^\infty(\mathcal{R}) \rightarrow \mathcal{R}$ onto the $n$th component and note that each $\varphi_n$ is completely positive and contractive.  Let $F_n \subseteq \mathcal{R}$ be an increasing sequence of finite dimensional factors with dense union and let $E_n : \mathcal{R} \rightarrow F_n$ be a trace-preserving expectation for each $n \geq 1$.  Each of the maps $\psi_n := E_n \circ \varphi_n : A \rightarrow F_n$ is nuclear since $F_n$ is finite dimensional.  Let $\psi : A \rightarrow \prod F_n \subseteq \ell^\infty(\mathcal{R})$ denote the direct product of the maps $\psi_n$ for $n \geq 1$.  Since $A$ is exact, $\psi$ is nuclear by Lemma 3.3 in \cite{Dadarlat:QDMorphisms}.  By construction, $\|\varphi_n(a) - \psi_n(a)\|_2 \rightarrow 0$ for all $a \in A$.  Hence if $q : \ell^\infty(\mathcal{R}) \rightarrow \mathcal{R}^\omega$ denotes the quotient map, then $\varphi = q \circ \tilde{\varphi} = q \circ \psi$.  As $\psi$ is nuclear, $\varphi$ is also nuclear.
\end{proof}

\begin{proof}[\textbf{Proof of Theorem \ref{thm:GabeTWW}}]
Suppose $\tau$ be a faithful, amenable trace on $A$.  Fix a trace preserving *-homomorphism $\varphi : A \rightarrow \mathcal{R}^\omega$ with a completely positive lift $A \rightarrow \ell^\infty(\mathcal{R})$.  By Lemma \ref{lem:Exactness}, $\varphi$ is nuclear, and since $\tau$ is faithful, $\varphi$ is also faithful.  By Theorem \ref{thm:LiftingTheorem}, there is a nuclear *-homomorphism $\psi : A \rightarrow \mathcal{Q}_\omega$ with $\pi \circ \psi = \varphi$.  In particular, $\psi$ is trace preserving and, by the Choi-Effros Lifting Theorem, $\psi$ has a completely positive lift $A \rightarrow \ell^\infty(\mathcal{Q})$.  Hence $\tau$ is quasidiagonal.
\end{proof}

\begin{acknowledgements}
In the first version of this paper posted on the arXiv, Theorem \ref{thm:LiftingTheorem} was proven (although not explicitly stated) when $A$ is nuclear and $\varphi$ is non-unital.  This is enough to prove Theorem \ref{thm:GabeTWW} in the case when $A$ is nuclear as stated in \cite{TikuisisWhiteWinter} by using the $2 \times 2$-matrix argument used in the proof of Theorem \ref{thm:LiftingTheorem} to reduce to the non-unital setting.  The modifications needed to extend the proof to the exact setting were suggested by Jamie Gabe.  The idea of using the classification of injective von Neumann algebras to produce a unital version of Theorem \ref{thm:LiftingTheorem} grew out of a conversation with Aaron Tikuisis and Stuart White.  I am also grateful to Tim Rainone for several helpful conversation related to this work and helpful comments on earlier versions of this paper.  Finally, I would like to thank the referee for many valuable comments which improved the exposition of this paper.
\end{acknowledgements}

\end{document}